\documentclass[11pt,twoside]{article}
\usepackage{latexsym}
\usepackage{amssymb,amsbsy,amsmath,amsfonts,amssymb,amscd,amsthm}
\usepackage{epsfig, graphicx}
\usepackage{color}
\usepackage{enumerate}
\usepackage{pdfsync}
\usepackage[T1]{fontenc}       % Police contenant les caracteres francais
\usepackage[francais, english]{babel}  % Placez ici une liste de langues, la derniere etant la langue principale
\usepackage[applemac]{inputenc} % LaTeX, comprends les accents !
\newcommand{\R}{\mathbb{R}}
\newcommand{\N}{\mathbb{N}}

\newcommand {\e}{\varepsilon}

\newcommand{\G}{\Gamma}

\newcommand{\caT}{{\cal T}}

\newcommand {\f}{\frac}
\newcommand {\pd}{\partial}

\newcommand{\beq}{\begin{equation}}
\newcommand{\beqa}{\begin{eqnarray}}
\newcommand{\bea} {\begin{array}{ll}}
\newcommand{\beqan}{\begin{eqnarray*}}
\newcommand{\eeq}{\end{equation}}
\newcommand{\eeqa}{\end{eqnarray}}
\newcommand{\eeqan}{\end{eqnarray*}}
\newcommand{\eea} {\end{array}}
\newtheorem{theorem}{Theorem}[section]
\newtheorem{lemma}[theorem]{Lemma}

\newtheorem{proposition}[theorem]{Proposition}

\newtheoremstyle{remarkb}
	{}
	{}
	{\normalfont}
	{}
	{\bfseries}
	{}
	{ }
	{}
\theoremstyle{remarkb}
\newtheorem{remark}[theorem]{Remark}
\numberwithin{equation}{section}
\newenvironment{acknowledgment}{\noindent{\bf Acknowledgment}}{}
\title{\Large \bf Parabolic models for chemotaxis on weighted networks}

\author{Fabio Camilli~$^{\text{a}}$,  Lucilla Corrias~$^{\text{b}}$
}
\date{\today}
\begin{document}
\maketitle
\pagestyle{plain}
\pagenumbering{arabic}

\begin{abstract}
In  this work we  consider the Keller-Segel model for chemotaxis on networks, both in the doubly parabolic case and in the parabolic-elliptic  one. Introducing  appropriate  transition conditions at   vertices, we prove the existence of a time global and spatially continuous solution   for each of the two systems. The main tool is the use of the explicit formula for the fundamental solution of the heat equation on a weighted graph and of the corresponding sharp estimates.
\end{abstract}

{\bf Key words:} Chemotaxis; network; transmission conditions; heat kernel.

{\bf AMS subject classification:} 92C17; 92C42; 35R02; 35Q92; 35A01.

%%%%%%%%%%%%%%%%%%%%%%%%%%%%%%%%%%%%%%%%%%%%
\section{Introduction}
\label{Sec:intro}
%%%%%%%%%%%%%%%%%%%%%%%%%%%%%%%%%%%%%%%%%%%%%
We  consider the classical Keller-Segel system for chemotaxis
\beq\label{KS}
\begin{split}
u_t&=\Delta u-\nabla\cdot(u\,\nabla c)\\
\e c_t&=\Delta c+u-\alpha\,c
\end{split}
\eeq
on a finite weighted network $\Gamma$, where $\e,\,\alpha \ge 0$.

System \eqref{KS} has been introduced in the early seventies in \cite{KS1,KS2} in order to model the aggregation phenomenon undergone by the slime mold \emph{Dictyostelium discoideum}. In this biological context, $u$ represents the cell concentration of the organism and satisfies the  continuity equation in \eqref{KS}, while $c$ is the chemo-attractant concentration  and solves    the diffusion equation in \eqref{KS}. In the Euclidean case,  i.e. when \eqref{KS}  is considered on a domain of $\R^d$, there is a vast literature on system \eqref{KS}. Depending on the space dimension $d$,    $\e>0$ (double parabolic case) or  $\e=0$ (parabolic-elliptic case) and the initial data   $u^0$, $c^0$, different phenomena can occur: global existence, finite or infinite time blow-up, peaks formation, threshold phenomena, etc.  We refer to \cite{CEM, HPainter,P} and the references therein for more details on that problem.

Even if the study of differential equations on networks goes back to the seminal papers by Lumer \cite{L1,L2} (see also the references in \cite{M1}), in the recent time there is an increasing interest in this type of problems in connection  with applications such as  data transmission,  traffic management, crowd motion and various problems in biology and neurobiology (\cite{DZ,GP,GN,M2,SBP}).

The parabolic model for chemotaxis \eqref{KS} on networks has been recently considered in \cite{BGKS} to describe the evolution of the ameboid organism \emph{Physarum polycephalum}. Indeed, during its evolution this slime mold arranges a network of thin tubes where nutrients and chemical signals are transmitted to the different parts of the organism (see \cite{NYT,TKN}). More precisely, in \cite{BGKS} system \eqref{KS} has been analyzed numerically and  the well posedness of a discrete system obtained via  its   finite differences approximation has been obtained. In \cite{GN} the authors consider a mathematical model in connection with the dermal wound healing process, where fibroblasts move along a polymeric scaffold to fill the wound, driven by chemotaxis. But in contrast with \eqref{KS}, in \cite{GN} the equations satisfied by the cells density $u$ are of hyperbolic type.

The goal of this  paper is to consider the parabolic system \eqref{KS} on a finite weighted network $\G$  composed of $m$ edges and to prove  the existence of a time global and spatially continuous solution $(u,c)$, in both  cases $\e>0$ and $\e=0$.  Therefore system \eqref{KS} shall appear to be formally equivalent to $m$ Keller-Segel systems, one on each of the $m$ edges, coupled via the transition conditions at the vertices of the network. As usual when dealing with differential equations on networks,
the transition conditions at the vertices shall play a crucial role. Coherently with the parabolic nature of the problem (see \cite{M2}), we look for a global continuous solution  on the whole network and consequently we prescribe the  continuity  of $u$ and $c$ at the vertices (see \eqref{contu}-\eqref{contc}). Moreover, we require for $u$   the  flux conservation  at the vertices, while for  $c$ a Kirchhoff type condition which guarantees the validity of the  maximum principle for diffusion equations on networks  (see \eqref{transu} and \eqref{transc} respectively). On the other hand, for simplicity reason, the network we consider has no boundary nodes. Our results can be however extended to the  case of Dirichlet or mixed boundary conditions.

More specifically, we consider solution of \eqref{KS} in the following integral sense
\begin{align}
&u(t,y)=P_tu^0(y)-\int_0^t P_{(t-s)}\pd_x(u(s)\pd_x c(s))(y)ds\,,\qquad\qquad y\in\G\,,  \label{u0} \\
&c(t,y)=e^{-{(\alpha/\e)}t}P_{(t/\e)}c^0(y)+\f1\e\, \int_0^te^{- ( \alpha/\e) (t-s)}P_{((t-s)/\e)}\,u(s)(y)ds\,,\qquad y\in\G\, ,\label{c0}
\end{align}
where $(P_t)_{t\ge0}$ is the semigroup generated by the laplacian $-\Delta_\G$ on $\G$ with domain the space of continuous functions on $\G$, belonging to $H^2$ on every edges of $\G$ and satisfying weighted Kirchhoff conditions at the nodes (see \eqref{operatordomain}). The interest in considering the formulation \eqref{u0}-\eqref{c0} lies in the fact that $(P_t)_{t\ge0}$ is given explicitly through the fundamental solution $H=H(t,x,y)$ (see formula \eqref{def:H} and \eqref{solheat}). Then, with this integral formula at hand, the proofs for local and global existence follow the corresponding  arguments in the Euclidean case, with however specific modifications due to the network structure.  In particular, it is not possible to use   known results about the existence of solutions of \eqref{KS} on a bounded interval $[0,L]$ with homogeneous Neumann boundary conditions (see \cite{HP} for instance). More than this, in order to get appropriate time bounds on the norm  of $u$ and $c$ we need to prove   optimal $L^p$  bounds  for the heat kernel $H$ on $\G$, which improve earlier  results in \cite{R,Cat1,Cat2}.

We remark that  the integral formulation for solutions of the heat equation on networks and the corresponding $L^p$ estimates can be the starting point to derive   qualitative information about the behavior of  the solutions with respect to the structural elements of the network.  We are pursuing this analysis in a forthcoming paper. Moreover they can be useful  not only for the    Keller-Segel model, but also to transpose other problems   of parabolic nature from the Euclidean case to the networks.

The paper is organized as follows. In Section \ref{Sec:preliminaries} we define the network $\G$, we recall the fundamental solution of the heat equation on $\G$  and we deduce the optimal $L^p$-estimates for the heat kernel. Section~\ref{Sec:KSNPP} is devoted to the study of the parabolic-parabolic chemotaxis system on $\G$, while Section \ref{Sec:KSNPE} to the parabolic-elliptic one. In the appendices we give the proofs of some technical results.
%%%%%%%%%%%%%%%%%%%%%%%%%%%%%%%%%%%%%%%%%%%%
\section{The heat equation on the network}
\label{Sec:preliminaries}
%%%%%%%%%%%%%%%%%%%%%%%%%%%%%%%%%%%%%%%%%%%%
This section is devoted to the definition of the network   and to the properties of the fundamental solution of the heat equation on that network. The fundamental solution has been computed by Roth \cite{R} for a finite network and generalized to the case of an infinite homogeneous tree and a countable graph by Cattaneo \cite{Cat1,Cat2}. Therefore, notations are coherent with the ones  in these papers (see also \cite{N1}). However, some properties, such as the optimal $L^1$ and $L^\infty$ time decay, are not contained in the cited papers. Their proofs are sketched in the Appendix \ref{Appendix:B}.
\subsection{The network}
\label{subsec:network}
Let consider a finite, connected and non-oriented (or undirected) {\sl network} $\Gamma$. This means that the underlying {\sl graph} ${\cal G}=(V,E)$ is defined through a non empty finite set of $n$ {\sl vertices} or {\sl nodes}, $V:=\{v_1,\dots,v_n\}$, a non empty finite set of $m$ non-oriented open {\sl edges} (or {\sl links}), $E:=\{e_1,\dots,e_m\}$, and that between every pair of nodes $v_i,v_j\in V$ there exists a path with edges in $E$. Furthermore, we assume that the graph has no self-loops (no edge connecting a vertex to itself). On the other hand, the graph can contain multiple links, i.e. the map $E\mapsto V\times V$ associating to each non-oriented edge its endpoints can be not injective.

Every edge may have a different length. However, we parametrize and normalize each $e_j\in E$ so that to identify $\overline e_j$ with the interval $[0,1]$. Since the network is undirected, every edge $e_j\in E$ can be parametrized in two different ways giving rise to two oriented edges $e_j^\pm$, i.e. there exist two  homeomorphism $\Pi_j^\pm:[0,1]\mapsto (\overline e_j)^\pm$, such that $\Pi_j^+(0)=\Pi_j^-(1)$ and $\Pi_j^+(1)=\Pi_j^-(0)$. We shall call an oriented edge an {\sl arc}, and we shall denote by $a_j$ any of the two edges $e_j^\pm$. Moreover, we shall denote by $-a_j$ the arc opposite to $a_j$ and the initial and terminal endpoints of $a_j$ by $I(a_j)$ and $T(a_j)$, respectively. We also denote by $E(v_i)$  the set of the index $j$ such that the edge $e_j$ has an endpoint at the vertex $v_i\in V$ and by $d(v_i)$  the {\sl degree} of $v_i$, i.e. the cardinality of $E(v_i)$.

Next, we define a {\sl path} $C$ on the network $\Gamma$ as a finite sequence of (at least two) arcs, $(a_{j_1}, \dots,a_{j_k})$, $k\ge2$, such that $T(a_{j_l})=I(a_{j_{l+1}})$, $l=1,\dots,k-1$. Thus a path is always oriented. We associate to each path $C=(a_{j_1}, \dots,a_{j_k})$ its length $|C|$ as the number of the arcs composing $C$. Then, given two points $x$ and $y$ on $\Gamma$, we shall note $C_k(x,y)$ the set of the paths of length $k$ such that $x$ belongs to the first arc of the path and $y$ belongs to the last arc of the path, i.e.
\[
C_k(x,y):=\left\{C=(a_{j_1}, \dots,a_{j_k})\ :\ x\in a_{j_1}\text{ and } y\in a_{j_k}\right\}\,,\qquad k=2,3,\dots\ .
\]
A {\sl geodesic path} joining $x$ to $y$ on $\Gamma$ is any path of minimum length in $\cup_{k\ge2}C_k(x,y)$. We shall denote ${\cal L}(x,y)$ the common length of any geodesic path joining $x$ to $y$  and we also define
\[
\rho(x,y):={\cal L}(x,y)-2\,.
\]

For every $x$ and $y$ belonging to the same $\overline e_j$, we define the distance $d(x,y)$ as
\[
d(x,y):=|(\Pi^\pm_j)^{-1}(x)-(\Pi^\pm_j)^{-1}(y)|\,,\qquad x,\, y\in\overline e_j\,.
\]
Then, for every $x$ and $y$ on $\Gamma$, we define the distance of $x$ to $y$ along $C=(a_{j_1}, \dots,a_{j_k})\in C_k(x,y)$ as
\[
d_C(x,y):=d(x,T(a_{j_1}))+d(y,I(a_{j_k}))+|C|-2\,.
\]
Obviously, $d_C(x,y)$ is symmetric with respect to $x$ and $y$, i.e. $d_C(x,y)=d_{-C}(y,x)$.

Finally, to each non-oriented edge $e_j\in E$ we associate a positive weight $\kappa(e_j)$ and we assume that
\[
0<\kappa_0\le\kappa(e_j)\le\kappa_1\,,\qquad \forall\ j=1,\dots,m\,.
\]
The weights $\kappa(e_j)$ shall influence the transmission or the reflection of $u$ through  the nodes. Indeed, for each couple of arcs $(a_i,a_j)$, we introduce the {\sl transfer/reflection coefficient} from $a_i$ to $a_j$ as
\beq\label{def:trancoef}
\epsilon_{(a_i\to a_j)}:=\left\{
\begin{split}
&2\,\kappa(e_i)(\sum_{l\in E(T(a_i))}\kappa(e_l))^{-1}\qquad\qquad\text{if }T(a_i)=I(a_j)\text{ and }a_j\ne -a_i\text{ (transmission)}\\
&2\,\kappa(e_i)(\sum_{l\in E(T(a_i))}\kappa(e_l))^{-1}-1\qquad\ \text{if }a_j=-a_i\text{ (reflection)}\\
&0\qquad\qquad\qquad\qquad\qquad\qquad\qquad\ \text{otherwise}
\end{split}
\right.
\eeq
The {\sl weight} $\epsilon(C)$ of a path $C=(a_{j_1}, \dots,a_{j_k})$ is then the product of the transfer/reflection coefficients of all the pairs of consecutive arcs composing~$C$, i.e.
\beq\label{def:epsC}
\epsilon(C):=\prod_{l=1}^{k-1}\epsilon_{(a_{j_l}\to a_{j_{l+1}})}\,.
\eeq

It is worth noticing that in case of reflection, the coefficient $\epsilon_{(a_i\to -a_i)}$ may be negative, and so also the weight $\epsilon(C)$ of all path $C$ passing through $a_i$ and $-a_i$ consecutively. Moreover, $\epsilon_{(a_i\to a_j)}\ne\epsilon_{(a_j\to a_i)}$ and $\epsilon(C)\neq\epsilon(-C)$, in general.

The definitions \eqref{def:trancoef} of the transfer/reflection coefficients and \eqref{def:epsC} of the paths weights  come naturally in the construction of the fundamental solution of the heat equation on the network given in Appendix~\ref{Appendix:A}.
\subsection{The heat equation on the network}
\label{subsec:semigroup}
A function $u$ defined on the network $\Gamma$ is a collection of $m$ functions $(u_j)_{j=1}^m$ such that $u_j:=u_{|\overline e_j}$. To every function $u$ on $\Gamma$ we associate the vector valued function $\tilde u=(\tilde u_1,\dots,\tilde u_m)$ defined on $[0,1]$ such that $\tilde u_j:=u\circ\Pi^\pm_j$. Then, we denote $u'_j(x)$ and $u''_j(x)$, $x\in e_j$, the derivatives $\tilde u'_j(\xi)$ and $\tilde u''_j(\xi)$ with respect to $\xi\in(0,1)$, $\xi=(\Pi^\pm_j)^{-1}(x)$. We also define the exterior normal derivative of $u_j$ at the endpoints of the arc $a_j$ as
\[
\f{\partial u_j}{\partial n}(I(a_j))=-\lim_{h\to0^+}\f{\tilde u_j(h)-\tilde u_j(0)}h\quad\text{and}\quad
\f{\partial u_j}{\partial n}(T(a_j))=\lim_{h\to0^-}\f{\tilde u_j(1+h)-\tilde u_j(1)}h\,.
\]

Next, we define the space of continuous functions on $\Gamma$
\[
C^0(\Gamma):=\{u=(u_j)_{j=1}^m\ :\ u_j(v_i)=u_k(v_i)\text{ if } j,k\in E(v_i)\,,\ i=1,\dots, n\}\,,
\]
the integral of a function $u$ over $\Gamma$
\[
\int_\Gamma u(x)dx:=\sum_{j=1}^m\kappa(e_j)\int_0^1\tilde u_j(\xi)\,d\xi\,,
\]
the Lebesgue spaces
\[
\begin{split}
L^p(\Gamma)&:=\{u=(u_j)_{j=1}^m\ :\ \|u\|^p_{L^p(\Gamma)}:=\sum_{j=1}^m\kappa(e_j)\|\tilde u_j\|^p_{L^p(0,1)}<\infty\}\,,\qquad p\in[1,\infty)\,,\\
L^\infty(\G)&:=\{u=(u_j)_{j=1}^m\ :\ \|u\|_{L^\infty(\Gamma)}:=\max_{1\le j\le m}\kappa(e_j)\|\tilde u_j\|_{L^\infty(0,1)}\}\,,
\end{split}
\]
and the Sobolev spaces
\[
W^{1,\infty}(\G):=\{u\in C^0(\G)\ :\ u'\in L^\infty(\G)\}\,,
\]
\[
H^r(\Gamma):=\{u\in C^0(\Gamma)\ :\ \|u\|^2_{H^r(\Gamma)}:=\sum_{j=1}^m\kappa(e_j)\|\tilde u_j\|^2_{H^r(0,1)}<\infty\}\,.
\]

With these notations, we can now introduce the operator $(D(-\Delta_\G),-\Delta_\G)$, where the domain $D(-\Delta_\G)$ is the set of the function $u$ in $H^2(\Gamma)$ satisfying the transmission conditions of Kirchhoff type at every vertex $v_i\in V$,
\beq\label{operatordomain}
D(-\Delta_\G):=\{u\in H^2(\Gamma)\ :\ \sum_{j\in E(v_i)}\kappa(e_j)\f{\partial u_j}{\partial n}(v_i)=0\,,\ i=1,\dots,n\}\,,
\eeq
and, for all $u\in D(-\Delta_\G)$, the laplacian $\Delta_\G$ on $\G$ is naturally defined as $\Delta_\G u=u''$. Then, $-\Delta_\G$ is densely defined in the Hilbert space $L^2(\G)$ endowed with the scalar product
\[
(u,v)_{L^2(\G)}=\sum_{j=1}^m\kappa(e_j)\int_0^1\tilde u'_j(\xi)\tilde v'_j(\xi)\,d\xi\,.
\]
It is also symmetric and positive and consequently accretive. Thanks to the transmission conditions, it can be also proved that $-\Delta_\G$ is $m$-accretive and therefore self-adjoint (see for instance \cite{KMS}). Hence, we can associate to $-\Delta_\G$ a semigroup of contractions on $L^2(\G)$, say $(\caT(t))_{t\ge0}$, whose generator is $\Delta_\G$. To conclude, given any $f\in L^2(\G)$, the function $u(t)=\caT(t)f$ is the unique solution of the heat equation
\beq\label{heat}
\left\{
\begin{array}{ll}
\pd_tu=\Delta_\G u\qquad & \text{on $(0,\infty)\times\G$, } \\
u(0)=f & \text{on $\G$. }
\end{array}
\right.
\eeq
in the space $C([0,\infty),L^2(\G))\cap C((0,\infty),D(-\Delta_\G))\cap C^1((0,\infty),L^2(\G))$.

Problem \eqref{heat} can be also written as a system of $m$ heat equations coupled through the continuity and transmission conditions at the vertex, i.e.
\beq\label{heat2}
\left\{
\begin{array}{ll}
\pd_t u_j =\pd_{xx} u_j\quad & \text{on }(0,\infty)\times e_j\,,\ j=1,\dots,m\\
u_j(0)=f_j\quad & \text{on }e_j\,,\ j=1,\dots,m\\
u_j(t,v_i)=u_k(t,v_i)\text{ if }j,k\in E(v_i)\,,\ i=1,\dots, n& t>0\quad\text{(continuity)}\\
\sum\limits_{j\in E(v_i)}\kappa(e_j)\f{\partial u_j}{\partial n}(t,v_i)=0\,,\  i=1,\dots,n& t>0\quad\text{(transmission condition)}
\end{array}
\right.
\eeq
These transmission conditions together provide, for each node $v_i$, a system of $d(v_i)$ equations for the $d(v_i)$ components $u_j$ of the solution $u$ such that $j\in E(v_i)$. Moreover, they reduces at the vertex $v_i$ of degree 1, $d(v_i)=1$,  to the homogeneous Neumann boundary condition.

Finally, it is worth noticing that the choice of the orientation of the edges $e_j$ has no consequences, since the heat equation \eqref{heat}-\eqref{heat2} and the problem \eqref{KS} are invariant under the transformation $\xi\to (1-\xi)$ that commute $\Pi^+_j$ into $\Pi^-_j$ and vice versa, as well as all the definitions given  above. On the other hand, orientation appears to be necessary for the construction of the fundamental solution of the heat equation on $\G$ below, that we shall use for the resolution of \eqref{KS}. For the analysis of \eqref{heat}-\eqref{heat2} through the  abstract semigroup method see \cite{KMS} and the reference therein.
\subsection{The fundamental solution of the heat equation on the network}
\label{subsec:H}
Let $G(t,z)=\f1{\sqrt{4\pi\,t}}e^{-\f{z^2}{4t}}$ denote the heat kernel on $(0,\infty)\times\R$, and consider the function defined on $(0,\infty)\times\Gamma\times\Gamma$ as
\beq\label{def:H}
H(t,x,y)=\delta_{i,j}\,\kappa^{-1}(e_i)\,G(t,d(x,y))+L(t,x,y)\,,\qquad x\in \overline e_i\,,\ y\in\overline e_j\,,\ i,j\in\{1,\dots,m\}\,,
\eeq
where $\delta_{ij}$ is the usual Kronecker's delta function and
\beq\label{def:L}
L(t,x,y)=\sum_{k\ge \rho(x,y)}\sum_{C\in C_{k+2}(x,y)}\kappa^{-1}(e_i)\,\epsilon(C)\,G(t,d_C(x,y))\,,\qquad x\in\overline e_i\,,\ y\in\overline e_j\,,\ i,j\in\{1,\dots,m\}\,.
\eeq

The first term in \eqref{def:H} is simply the restriction of the fundamental solution of the heat equation on each edge of the network. The second term is determined in such a way that the function $H$ satisfies the continuity and transmission conditions in \eqref{heat2} with respect to $y$, for any fixed $x\in\Gamma$ (see Appendix~\ref{Appendix:A}).
%By symmetry (see Theorem \ref{th:R}), $H$ is also a solution of \eqref{heat2} with respect to $x$, for any fixed $y\in\Gamma$.
%Therefore, it takes into account the influence that the edge $e_j$ has on the edge $e_i$ and vice versa.
More specifically, since the network is composed of $m$ edges, it holds, for all $k\in\N$\,, that
\beq\label{est:card}
\text{card}(C_{k+2}(x,y))\le 2\,\left(\max_{i=1,\dots,n}d(v_i)\right)^{k+1}\le 2\,m^{k+1}\,.
\eeq
Hence, the sum with respect to the paths $C\in C_{k+2}(x,y)$ in \eqref{def:L} is finite. We also have that the coefficients \eqref{def:trancoef} are uniformly bounded : $|\epsilon_{(a_i\to a_j)}|\le2\kappa_1\,\kappa_0^{-1}:=\overline\epsilon$, $i,j=1,\dots,m$. Therefore, by \eqref{est:card}, we get
\beq\label{est:L}
|L(t,x,y)|\le \kappa_0^{-1}\sum_{k=0}^{+\infty}(\overline\epsilon\,m)^{k+1}\,\f{e^{-k^2/4t}}{\sqrt{\pi\,t}}<\infty\,.
\eeq
%We have also that if $x$ and $y$ belong to the same edge ($e_i=e_j$ in \eqref{def:H}-\eqref{def:L}), $L(t,x,y)=0$ and the function $H$ reduced to the heat kernel $G$ on that edge (to be proved.....).
The latter estimate implies that the series giving $L(t,x,y)$ is normally convergent over $[t_1,t_2]\times\Gamma\times\Gamma$, for any fixed $t_1,t_2>0$. Therefore, the associated vector valued function $\tilde H=\tilde H(t,\xi,\eta)$ is continuous with respect to $(t,\xi,\eta)\in(0,\infty)\times[0,1]\times[0,1]$, component by component. Similarly, for any fixed $\xi\in(0,1)$, the derivatives $\partial_t\tilde H$, $\partial_\eta\tilde H$ and $\partial_{\eta\,\eta}\tilde H$ exist and are continuous with respect to $(t,\xi,\eta)\in(0,\infty)\times(0,1)\times(0,1)$. They can be computed differentiating under the sum sign and $\tilde H$ satisfies the heat equation $\partial_t\tilde H=\partial_{\eta\,\eta}\tilde H$, component by component.
%Furthermore, because of the non-oriented nature of the network and of the symmetry of the heat kernel $G$, $\tilde H$ is symmetric with respect to $\eta$ and $\xi$. Therefore, the previous properties hold true with respect to the $\xi$ variable, for any fixed $\eta\in[0,1]$.
These and other properties of the function $H$ are resumed below.

\begin{theorem}[\cite{R}]\label{th:R}
Let $H$ be the function defined in \eqref{def:H}. Then,
\begin{enumerate}
\item[(i)] $H$ is continuous on $(0,\infty)\times\G\times\G$;
\item[(ii)] $\pd_tH(t,x,y)$ exists for all $(t,x,y)\in(0,\infty)\times\G\times\G$ and it is continuous on $(0,\infty)\times\G\times\G$;
\item[(iii)] the derivatives $\pd_\eta\tilde H(t,\xi,\eta)$ and $\pd_{\eta\,\eta}\tilde H(t,\xi,\eta)$, exist for all $(t,\xi,\eta)\in(0,\infty)\times(0,1)\times(0,1)$ and  are continuous on $(0,\infty)\times(0,1)\times(0,1)$;
\item[(iv)] $H(t,x,\cdot)\in D(-\Delta_\G)$ for all $(t,x)\in (0,\infty)\times\G$;
\item[(v)] $\pd_tH(t,x,y)=\pd_{yy}H (t,x,y)$ for all $(t,x,y)\in (0,\infty)\times\G\times\G$;
\item[(vi)]  for all $f\in C^0(\G)$, $\int_\G H(t,x,y)f(x)dx\to f(y)$ for $t\to 0^+$, uniformly with respect to $y\in \G$;
\item[(vii)] for all $f\in C^0(\G)$, the function
          \beq\label{solheat}
           P_tf(y):=\int_\G H(t,x,y)f(x)dx, \qquad (t,y)\in(0,\infty)\times\G
          \eeq
        with $P_0f=f$ is the unique continuous solution of the initial valued problem \eqref{heat}.
\end{enumerate}
Moreover, $H$ is symmetric with respect to $x,y\in\G$, i.e. $H(t,x,y)=H(t,y,x)$ for all $t\in (0,\infty)$ and the properties above hold true with respect to $x$, for any fixed $y$.
\end{theorem}
The function $H$ is also the unique function satisfying properties $(i)$-$(vii)$ in Theorem \ref{th:R}.
As observed in \cite{R}, $(P_t)_{t\ge0}$ is a strongly continuous semigroup on $L^2(\G)$, whose infinitesimal generator is the closure of $-\Delta_\G$ in $L^2(\G)$. It is obviously the same semigroup determined in \cite{KMS} by variational methods.

 It is worth noticing that $H$ is not a priori positive since the weights $\epsilon(C)$ could be negative. Furthermore, the spatial symmetry of $H$ is due to the symmetry of $G$ and to fact that changing $x$ with $y$ in \eqref{def:H}-\eqref{def:L}, the path $C$ changes into $-C$ and $\kappa^{-1}(e_i)\,\epsilon(C)=\kappa^{-1}(e_j)\,\epsilon(-C)$, (see \eqref{def:trancoef}-\eqref{def:epsC}).  The construction of $H$ has been done in \cite{R} in the case $\kappa(e_j)=1$, $\forall j$. The generalization to the case of a weighted graph has been considered in \cite{Cat1,Cat2}, where however the construction is not detailed. Therefore, we give the general construction in the Appendix \ref{Appendix:A}, for the reader convenience.

We close this section showing the optimal decay in time of $H$ and its derivatives. The proofs are sketched in the Appendix \ref{Appendix:B}.
\begin{proposition}\label{prop:proprietaH}
Let $H$ be defined as in  \eqref{def:H}. Then,
\beq\label{int H}
\int_\G H(t,x,y)dy=1\,,\qquad\forall\ (t,x)\in(0,\infty)\times\Gamma\,,
\eeq
and there exist constants $C_i>0$, $i=1,\dots,4$, such that for all $t>0$ it holds
\beq\label{est:norme1H}
\sup_{x\in\G}\|H(t,x,\cdot)\|_{L^1(\G)}\le C_1\,,
\eeq
\beq\label{est:normeinftyH}
\|H(t)\|_{L^\infty(\G\times\G)}\le C_2(1+t^{-1/2})\,,
\eeq
\beq\label{est:norme1derH}
\sup_{x\in\G}\|\partial_yH(t,x,\cdot)\|_{L^1(\G)}+\sup_{y\in\G}\|\partial_yH(t,\cdot,y)\|_{L^1(\G)}\le C_3(1+t^{-1/2})\,,
\eeq
\beq\label{est:normeinftydevH}
\|\partial_yH(t)\|_{L^\infty(\G\times\G)}\le C_4(1+t^{-1})\,.
\eeq
Moreover, since $H$ is symmetric with respect to $x$ and $y$, all the above properties hold true changing $x$ with $y$.
\end{proposition}
%
%%%%%%%%%%%%%%%%%%%%%%%%%%%%%%%%%%%%%%%%%%%%
\section{The parabolic-parabolic Keller-Segel system on the network}
\label{Sec:KSNPP}
%%%%%%%%%%%%%%%%%%%%%%%%%%%%%%%%%%%%%%%%%%%%
According to the notations and definitions of the previous section, system \eqref{KS}, endowed with the natural continuity and transmission conditions, can be written on the network $\G$ as following
\begin{eqnarray}
\label{KSunet}
&&\pd_t u_j =\pd_{yy} u_j-\pd_y( u_j\,\pd_y c_j)\qquad\qquad\qquad\qquad\qquad\ \text{on }(0,\infty)\times e_j\,,\ j=1,\dots,m,\\
\label{KScnet}
&&\e\,\pd_tc_j =\pd_{yy} c_j+u_j-\alpha\, c_j\qquad\qquad\qquad\qquad\qquad\ \text{ on }(0,\infty)\times e_j\,,\ j=1,\dots,m,\\
\label{initial}
&&u_j(0,y)=u_j^0(y)\quad\text{and}\quad c_j(0,y)=c_j^0(y)\,,\qquad\qquad  y\in\G\,,\ j=1,\dots,m,\\
\label{transu}
&&\sum\limits_{j\in  E(v_i)}\kappa(e_j)\f{\partial u_j}{\partial n}(t,v_i)=0\,,\qquad\qquad\qquad\qquad\qquad\quad\ t>0\,,\ i=1,\dots,n,\\
\label{transc}
&&\sum\limits_{j\in  E(v_i)}\kappa(e_j)\f{\partial c_j}{\partial n}(t,v_i)=0,\qquad\qquad\qquad\qquad\qquad\quad\ t>0\,,\ i=1,\dots,n,\\
\label{contu}
&&u_j(t,v_i)=u_k(t,v_i)\text{ if }j,k\in E(v_i)\,,\qquad\qquad\qquad\qquad t>0\,,\ i=1,\dots,n,\\
\label{contc}
&&c_j(t,v_i)=c_k(t,v_i)\text{ if }j,k\in E(v_i)\,,\qquad\qquad\qquad\qquad\ t>0\,,\ i=1,\dots,n\,.
\end{eqnarray}
As for problem \eqref{heat2}, there is no coupling among the $m$ systems \eqref{KSunet}-\eqref{KScnet}-\eqref{initial} on each edge $e_j$. The systems are coupled only through the transmission conditions \eqref{transu} and \eqref{transc}, which express the conservation of the flux at the vertices for $u$ and $c$ (Kirchhoff condition), and through the continuity conditions \eqref{contu} and \eqref{contc}. Again, conditions \eqref{transu}-\eqref{contc} give exactly $2d(v_i)$ equations for the $2d(v_i)$ functions $u_j,c_j$ such that $j\in E(v_i)$. Furthermore, conditions \eqref{transu} and \eqref{transc}, together with the continuity of $u$, guarantee  the conservation of the initial mass
\beq\label{consmass}
\int_\G u(t,y)\,dy=\int_\G u^0(y)\,dy=:M\,,\qquad t>0\,.
\eeq
Indeed
\[
\begin{split}
\f d{dt}\int_\G u(t,y)\,dy&=\sum_{j=1}^m\kappa(e_j)\f d{dt}\int_0^1\tilde u_j(t,\eta)\,d\eta=
\sum_{j=1}^m\kappa(e_j)\int_0^1(\partial_{\eta\eta}\tilde u_j(t,\eta)-\partial_\eta(\tilde u_j(t,\eta)\partial_\eta\tilde c_j(t,\eta)))d\eta\\
&=\sum_{j=1}^m\kappa(e_j)[\partial_{\eta}\tilde u_j(t,\eta)-\tilde u_j(t,\eta)\partial_\eta\tilde c_j(t,\eta)]_0^1
=\sum_{i=1}^n\sum_{j\in E(v_i)}\kappa(e_j)[\f{\partial u_j}{\partial n}-u_j\f{\partial c_j}{\partial n}](t,v_i)=0\,.
\end{split}
\]

\begin{remark}
A different but again natural condition that one can impose at the vertices instead of \eqref{transu} is the conservation of the total flux
\[
\sum\limits_{j\in  E(v_i)}\kappa(e_j)[\f{\partial u_j}{\partial n}-u_j\,\f{\partial c_j}{\partial n}](t,v_i)=0\,,\qquad t>0,\ i=1,\dots,n\,.
\]
However, the latter together with the continuity of $u$ at the vertices and the Kirchhoff condition \eqref{transc} imply \eqref{transu}.
\end{remark}

In this section we assume $\e>0$ and we consider solutions of the Keller-Segel system in the integral form \eqref{u0}-\eqref{c0}. Then, for $f$ integrable over $\G$, we introduce the notation
\[
(H(t)*f)(y):=\int_\G H(t,x,y)f(x)dx\,,\qquad y\in\G\,.
\]
Thanks to the continuity of the heat kernel $H$ on $\G$, if $u$ is continuous on $\G$ and $c$ satisfies the Kirchhoff condition \eqref{transc}, equations \eqref{u0}-\eqref{c0} read equivalently~as
\begin{align}
&u(t,y)=(H(t)*u^0)(y)+\int_0^t (\pd_xH(t-s)*(u(s)\pd_xc(s)))(y)ds\,,\qquad\qquad y\in\G\,,                  \label{u2} \\
&c(t,y)=e^{-{(\alpha/\e)}t}(H(t\,\e^{-1})*c^0)(y)+ \f1\e\int_0^te^{- ( \alpha/\e) (t-s)}(H((t-s)\e^{-1})*u(s))(y)ds\,,\qquad y\in\G\,.           \label{c2}
\end{align}
It is worth noticing that thanks to property \eqref{int H}, any integral solutions \eqref{u2}-\eqref{c2} satisfies the mass conservation \eqref{consmass}.

\begin{theorem}[Local existence]\label{th:KSlocalexist}
Let $\e>0$, $\alpha\ge0$ and assume $u^0\in L^\infty(\G)$, $c^0\in W^{1,\infty}(\G)$. Then, there exist $T=T(\|u^0\|_{L^\infty(\G)},\|\partial_xc^0\|_{L^\infty(\G)},\e)>0$ and a unique integral solution \eqref{u2}-\eqref{c2} of the Keller-Segel system with
\[
u\in L^\infty((0,T);C^0(\G))\,,\quad c\in L^\infty(0,T;W^{1,\infty}(\G))\,,
\]
satisfying the transmission conditions \eqref{transu} and \eqref{transc} and the mass conservation \eqref{consmass}.
\end{theorem}
\begin{proof}
For $u_0$ given, $A:=\|u^0\|_{L^\infty(\G)}$, $K:=\sup_{t>0,y\in\G}\|H(t,\cdot,y)\|_{L^1(\G))}>1$ and $T>0$ to be chosen later, let
\[
B:=\{u\in L^\infty((0,T)\times\G):\,u(0,y)=u^0(y) \text{ and }\,\sup_{0\le t<T}\|u(t)\|_{L^\infty(\G)}\le K\,A+1\}
\]
and $d(u_1,u_2):=\sup_{0\le t<T}\|u_1(t)-u_2(t)\|_{L^\infty(\G)}$.

Next, for $u\in B$ fixed, $c_0$ given and $c$ defined through $u$ by \eqref{c2}, we define on $B$ the map
\beq\label{mapPsi}
\Psi(u)(t,y):=(H(t)*u^0)(y)+\int_0^t (\pd_xH(t-s)*(u(s)\pd_xc(s)))(y)ds\,,\qquad (t,y)\in(0,T)\times\G\,.
\eeq
Since $(B,d)$ is a non empty complete metric space, we shall prove the claimed local existence using the Banach fixed point theorem.

{\sl First step : $\Psi(B)\subset B$. } From \eqref{mapPsi} we have
\beq\label{est:mapPsi}
|\Psi(u)(t,y)|\le K\,A+\kappa_0^{-2}\sup_{0\le s<T}\|u(s)\|_{L^\infty(\G)}\int_0^t \|\partial_xH(t-s,\cdot,y)\|_{L^1(\G)}\|\pd_x c(s)\|_{L^\infty(\G)}ds\,,\quad y\in\G\,.
\eeq
Next, owing to the property $\partial_yH(t,x,y)=-\partial_xH(t,x,y)$, by \eqref{c2} we get for any $y\in\G$
\beq\label{est:gradc}
\begin{split}
\partial_yc(t,y)&=e^{-{(\alpha/\e)}t}(\partial_yH(t\,\e^{-1})*c^0)(y)+\f1\e\int_0^te^{- ( \alpha/\e) (t-s)}(\partial_yH((t-s)\e^{-1})*u(s))(y)ds\\
&=-e^{-{(\alpha/\e)}t}(\partial_xH(t\,\e^{-1})*c^0)(y)+\f1\e\int_0^te^{- ( \alpha/\e) (t-s)}(\partial_yH((t-s)\e^{-1})*u(s))(y)ds\,.
\end{split}
\eeq
Furthermore, we observe that embedding the non-oriented network $\G$ into the oriented one \break$(V,\{e^\pm_j;j=1,\dots,m\})$, denoted by $2\G$, by construction the fundamental solution $H$ of the heat equation \eqref{heat2} on $\G$ is also solution on $2\G$ and $\int_{2\G}H(t,x,y)f(x)dx=2\int_{\G}H(t,x,y)f(x)dx$, for any $f$ integrable on $\G$ (and so on $2\G$). Therefore, for any $y\in\G$,
\[
\begin{split}
\int_\G\partial_x&H(t\,\e^{-1},x,y)c^0(x)\,dx=\f12\int_{2\G}\partial_xH(t\,\e^{-1},x,y)c^0(x)\,dx\\
&=\f12\sum_{j=1}^m\kappa(e_j)\int_0^1\partial_\xi \tilde H_j(t\,\e^{-1},\xi,\eta)\tilde c^0_j(\xi)\,d\xi
+\f12\sum_{j=1}^m\kappa(e_j)\int_0^1\partial_\xi \tilde H_j(t\,\e^{-1},1-\xi,\eta)\tilde c^0_j(1-\xi)\,d\xi\\
&=\f12\sum_{j=1}^m\kappa(e_j)[\tilde H_j(t\,\e^{-1},\xi,\eta)\tilde c^0_j(\xi)]_0^1-\f12\sum_{j=1}^m\kappa(e_j)\int_0^1\tilde H_j(t\,\e^{-1},\xi,\eta)\partial_\xi\tilde c^0_j(\xi)\,d\xi\\
&\quad+\f12\sum_{j=1}^m\kappa(e_j)[\tilde H_j(t\,\e^{-1},1-\xi,\eta)\tilde c^0_j(1-\xi)]_0^1+\f12\sum_{j=1}^m\kappa(e_j)\int_0^1\tilde H_j(t\,\e^{-1},1-\xi,\eta)\partial_\xi\tilde c^0_j(1-\xi)\,d\xi\\
&=-(H(t\,\e^{-1})*\partial_xc^0)(y)\,,
\end{split}
\]
and \eqref{est:gradc} becomes
\[
\partial_yc(t,y)=-e^{-{(\alpha/\e)}t}(H(t\,\e^{-1})*\partial_xc^0)(y)+\f1\e\int_0^te^{- ( \alpha/\e) (t-s)}(\partial_yH((t-s)\e^{-1})*u(s))(y)ds\,.
\]
%So that, by \eqref{est:norme1H} and \eqref{est:normeinftyH}
%\beq\label{est:terminpartialc}
%|\int_\G\partial_xH(t,x,y)c^0(x)\,dx|\le C(1+t^{-1/2})\max_{i=1,\dots,n;\ j\in E(v_i)}|c^0_j(v_i)|+K\|\partial_xc^0\|_{L^\infty(\G)}\,.
%\eeq
Using \eqref{est:norme1H} we arrive at the following estimate for the spatial derivative of $c$
\[
|\pd_y c(t,y)|\le \kappa_0^{-1}K\|\partial_xc^0\|_{L^\infty(\G)}+(\e\,\kappa_0)^{-1}\sup_{0\le s<T}\|u(s)\|_{L^\infty(\G)}\int_0^t\|\partial_yH((t-s)\e^{-1},\cdot,y)\|_{L^1(\G)}ds\,,\qquad y\in\G\,,
\]
and by \eqref{est:norme1derH}
\beq\label{est:gardc}
\|\pd_y c(t)\|_{L^\infty(\G)}\le \kappa_0^{-1}K\|\partial_xc^0\|_{L^\infty(\G)}+C(K\,A+1)(\e^{-1}t+\e^{-\f12}t^{\f12})\,,
\eeq
where $C>0$ does not depend on $\e$. Finally, plugging \eqref{est:gardc} into \eqref{est:mapPsi} and using the decaying properties of $H$ again, we get for $t\in(0,T)$
\[
\begin{split}
\|\Psi(u)(t)\|_{L^\infty(\G)}&\le K\,A+C(K\,A+1)\int_0^t(1+(t-s)^{-1/2})\|\pd_x c(s)\|_{L^\infty(\G)}ds\\
&\le K\,A+\widetilde C(t+t^{1/2})(1+\e^{-1}t+\e^{-\f12}t^{\f12})\,,
\end{split}
\]
where $\widetilde C=\widetilde C(K,A,\G,\|\pd_x c^0\|_{L^\infty(\G)})$. Therefore, for $T=T(\|u^0\|_{L^\infty(\G)},\|\partial_xc^0\|_{L^\infty(\G)},\e)>0$ sufficiently small, it holds
\[
\sup_{0\le t<T}\|\Psi(u)(t)\|_{L^\infty(\G)}\le K\,A+1\,.
\]
To obtain the claim, we also observe that $\Psi(u)(0,y)=u^0(y)$ since by definition $H(0)*u^0=u^0$.

{\sl Second step : $\Psi$ is a contraction map on $B$.}
Let $u_1,u_2\in B$. By \eqref{mapPsi} and arguing as in the previous step, we get for all $(t,y)\in(0,T)\times\G$
\beq\label{est:diffPsi}
\begin{split}
|\Psi(u_1)-\Psi(u_2)|(t,y)&\le\int_0^t|\pd_xH(t-s)*[(u_1-u_2)\pd_x c_1+u_2(\pd_xc_1-\pd_xc_2)](s)|(y)ds\\
&\le d(u_1,u_2)\kappa_0^{-2}\int_0^t \|\partial_xH(t-s,\cdot,y)\|_{L^1(\G)}\|\pd_x c_1(s)\|_{L^\infty(\G)}ds\\
&\quad +(K\,A+1)\kappa_0^{-2}\int_0^t \|\partial_xH(t-s,\cdot,y)\|_{L^1(\G)}\|(\pd_x c_1-\pd_xc_2)(s)\|_{L^\infty(\G)}ds\,,
\end{split}
\eeq
and for all $t\in(0,T)$
\beq\label{est:diffc}
\|(\pd_yc_1-\pd_yc_2)(t)\|_{L^\infty(\G)} \le C\,d(u_1,u_2)(\e^{-1}t+\e^{-\f12}t^{\f12})\,.
\eeq
Plugging \eqref{est:gardc} and \eqref{est:diffc} into \eqref{est:diffPsi} and using \eqref{est:norme1derH}, we arrive at
\[
\|(\Psi(u_1)-\Psi(u_2))(t)\|_{L^\infty(\G)}\le \widetilde C\,d(u_1,u_2)(t+t^{1/2})(1+\e^{-1}t+\e^{-\f12}t^{\f12})\,.
\]
Hence,  for $T$ sufficiently small again, $\Psi$ is a contraction on $B$.

{\sl Third step : conclusion.}   By the previous steps, it follows that there exists a unique fixed point $u\in B$ of $\Psi$ and that $(u,c)$ satisfies the integral system  \eqref{u2}-\eqref{c2}. Furthermore, $c$ is continuous on $\G$ and satisfies the transmission condition \eqref{transc}  because $H$ is continuous on $\G$ and satisfies the same condition. Again because of the regularity of $H$, $u$ is differentiable on $\G$ and $c$ is twice differentiable. Consequently, performing an integration by part on each edge in the second term of the r.h.s. of \eqref{u2}, $u$ can be also written as
\[
\begin{split}
u(t,y)=(H(t)*u^0)(y)&+\int_0^t\sum_{i=1}^nH(t-s,v_i,y)\sum_{j\in E(v_i)}\kappa(e_j)u_j(s,v_i)\f{\partial c_j}{\partial n}(s,v_i)\,ds\\
&-\int_0^t(H(t-s)*\pd_x(u(s)\pd_xc(s)))(y)ds\,,
\end{split}
\]
implying that $u(t)\in C^0(\G)$ holds true for all $t\in(0,T)$. Finally, the continuity of $u$ together with \eqref{transc} gives that
\[
u(t,y)=(H(t)*u^0)(y)-\int_0^t(H(t-s)*\pd_x(u(s)\pd_xc(s)))(y)ds\,.
\]
So that $u$ satisfies \eqref{transu} and the proof is complete.
\end{proof}

We conclude this section showing the existence of a classical solution of system \eqref{KSunet}-\eqref{contc} in $(0,T)$ for any $T>0$, i.e. we do not exclude that the solution blow-up for $T\to +\infty$. More specifically, we shall prove the following.
\begin{theorem}[Global existence and positivity]\label{th:KSglobalexist}
Under the hypothesis of Theorem \ref{th:KSlocalexist}, for all $T>0$ there exists a solution $(u,c)$ of the Keller-Segel system on the time interval $[0,T]$. Moreover, if the initial data $u^0$ and $c^0$ are nonnegative, the solution $(u,c)$ is nonnegative.
\end{theorem}
\begin{proof}
The global existence result is obtained by extending the local in time solution obtained in Theorem \ref{th:KSlocalexist}. Indeed, let $T_{max}$ be the maximal time of existence of the obtained local solution. Then, the limits as $t\to T_{max}^-$ of $u$ and $c$ exist  and depend  only on $\|u^0\|_{L^\infty(\G)}$, $\|\partial_xc^0\|_{L^\infty(\G)}$ and $\e$. Therefore, it is possible to extend $(u(t),c(t))$ behind $T_{max}$, iteratively as many time as it is necessary to reach $T>0$.

In order obtain the positivity of the solution $(u,c)$ when the initial data are positive, we analyze the time evolution of $\int_\G\phi(u(t,y))\,dy$, where $\phi$ is a smooth function on $\R$ such that $\phi(z)>0$ if $z<0$, $\phi(z)=0$ if $z\ge0$ and there exists $C>0$ such that $0\le\phi''(z)z^2\le C\phi(z)$, for all $z\in\R$. Owing to \eqref{KSunet} and to the Kirchhoff conditions \eqref{transu} and \eqref{transc}, we have for any $\delta>0$
\[
\begin{split}
\f d{dt}\int_\G\phi(u(t,y))\,dy&=\sum_{j=1}^m\kappa(e_j)\int_0^1\phi'(\tilde u_j(t,\eta))(\partial_{\eta\eta}\tilde u_j-\partial_\eta(\tilde u_j\partial_\eta\tilde c_j))(t,\eta)d\eta\\
&\le-\sum_{j=1}^m\kappa(e_j)\int_0^1\phi''(\tilde u_j)(\partial_\eta\tilde u_j)^2(t,\eta)d\eta\\
&\qquad+\sum_{j=1}^m\kappa(e_j)\|\partial_\eta\tilde c_j(t)\|_{L^\infty(0,1)}\int_0^1\phi''(\tilde u_j)|\tilde u_j||\partial_\eta\tilde u_j|(t,\eta)d\eta\\
&\le(\f\delta2-1)\int_\G\phi''(u)(\partial_yu)^2(t,y)dy+\f{\kappa_0^{-1}}{2\,\delta}\|\partial_yc(t)\|^2_{L^\infty(\G)}\int_\G\phi''(u)u^2(t,y)\,dy\,.
\end{split}
\]
Choosing $\delta<2$, by the properties of $\phi$ we get the differential inequality
\[
\f d{dt}\int_\G\phi(u(t,y))\,dy\le \f{C\,\kappa_0^{-1}}{2\,\delta}\|\partial_yc(t)\|^2_{L^\infty(\G)}\int_\G\phi(u(t,y))\,dy.
\]
Applying the Gronwall lemma, we obtain that $\phi(u(t,y))=0$, so that $u(t,y))\ge0$.

The positivity of $c$ does not follow from \eqref{c2}, since $H$ is not a priori positive, as observed before. Instead, it follows from the maximum principle for parabolic equations on network \cite{JvB2}, taking also into account that $u$ is positive.
\end{proof}

\begin{remark}[Energy]
  As for the euclidian case, solutions of the Keller-Segel system \eqref{KS} on $\G$ that satisfy the continuity and transmission conditions \eqref{transu}-\eqref{contc}, satisfy also the energy dissipation equation
\beq\label{energydissipation}
\f d{dt}\mathcal E(u(t),c(t))=-\int_\G u(t,x)|\partial_x(\log u-c)|^2(t,x)\,dx-\e\int_\G(\partial_tc(t,x))^2dx\,,
\eeq
where $\mathcal E$ is the usual free energy associated to the Keller-Segel system, i.e.
\[
\mathcal E(u,v):=\int_\G u\log u\,dx-\int_\G u\,c\,dx+\f12\int_\G|\partial_x c|^2dx+\f\alpha2\int_\G c^2\,dx\,.
\]
In particular, the global solution of Theorem \ref{th:KSglobalexist} satisfies \eqref{energydissipation}.
\end{remark}
%%%%%%%%%%%%%%%%%%%%%%%%%%%%%%%%%%%%%%%%%%%%
\section{The parabolic-elliptic Keller-Segel system on the network}
\label{Sec:KSNPE}
%%%%%%%%%%%%%%%%%%%%%%%%%%%%%%%%%%%%%%%%%%%%
We shall consider in this section the parabolic-elliptic system \eqref{KS}, i.e. $\e=0$, not only for the sake of completeness, but also to put in evidence the different behaviour of the two systems on a network: in contrast with the parabolic-parabolic case of the previous section, here positives solutions can not exhibit blow-up as $t\to+\infty$.

The integral solution $(u,c)$ in \eqref{u2}-\eqref{c2}, reads now as
\beq\label{u3}
u(t,y)=(H(t)*u^0)(y)+\int_0^t (\pd_xH(t-s)*(u(s)\pd_xc(s)))(y)ds\,,\qquad\qquad y\in\G\,,
\eeq
where $c$ is the weak solution of the elliptic equation on $\G$
\beq\label{c3}
\begin{array}{ll}
-\pd_{yy}c_j=u_j-\alpha\,c_j\qquad\qquad\qquad&  y\in  e_j\,,\ j=1,\dots,m\,,\\
\quad c_j(t,v_i)=c_k(t,v_i) &j,k\in E(v_i)\,,\ i=1,\dots,n\,,\\
\ \displaystyle\sum_{j\in  E(v_i)} \kappa(e_j)\f{\partial c_j}{\partial n}(t,v_i)=0,& i=1,\dots,n\,.
\end{array}
\eeq

Elliptic equations on networks are studied in \cite{N3}. Following \cite{N3}, we can state the existence result below. We sketch some ideas of the proof in the Appendix for the reader's convenience.
\begin{proposition}\label{prop:elliptic}
Let $\alpha>0$. Given $z\in L^\infty(\G)$, the elliptic problem
\beq\label{elliptic}
\begin{array}{ll}
- w''_j=z_j-\alpha\, w_j\qquad\qquad\qquad&  x\in  e_j\,,\ j=1,\dots,m,\\
\quad w_j(v_i)=w_k(v_i) &j,k\in E(v_i)\,,\ i=1,\dots,n\,,\\
\quad\sum_{j\in  E(v_i)} \kappa(e_j)\f{\partial w_j}{\partial n}(v_i)=0,& i=1,\dots,n\,,
\end{array}
\eeq
admits a unique weak solution $w\in H^1(\G)$. Moreover $w\in W^{1,\infty}(\G)$ and there exists $c=c(\G,\alpha)>0$ such that
\begin{align}
&c(\G,\alpha)\inf_\G\{z\}\le w(x)\le c(\G,\alpha)\sup_\G\{z\}\,,\qquad\qquad\forall\ x\in\G\,,\label{maxprinc}\\
&\quad\quad\|w'\|_{L^\infty(\G)}\le\|z\|_{L^\infty(\G)}\,.\label{estderiv}
\end{align}
Finally if $z\in C^0(\G)$, then $w\in C^2(\G):=\{u\in C^0(\G):\, \tilde u_j\in C^2((0,1)),\,j=1,\dots,m\}$.
\end{proposition}

Our main result of this section is the following global existence theorem.
\begin{theorem}[Global existence]\label{th:elliptic globalexistence}
Let $\e=0$, $\alpha>0$ and assume $u^0\in L^\infty(\G)$, $u^0\ge0$. Then, system \eqref{u3}-\eqref{c3} has a global in time integral-weak solution $(u,c)$ such that $u\ge0$, $c\ge0$ and
\[
u\in L^\infty((0,\infty);C^0(\G))\,,\quad c\in L^\infty((0,\infty);C^2(\G))\,.
\]
Moreover, the conservation of mass \eqref{consmass} and the transmission conditions \eqref{transu} and \eqref{transc} are also satisfied.
\end{theorem}

\begin{proof}
The local existence result can be proved as in Theorem \ref{th:KSlocalexist}, simply replacing the estimate \eqref{est:gardc} with the estimate \eqref{estderiv}.

It is worth emphasizing that the nonnegativity of the initial data $u_0$ is not necessary for the local existence. On the other hand, if $u^0\ge0$, then the positivity of the solution can be proved as for the corresponding result in Theorem \ref{th:KSglobalexist}.

To prove the global existence of a nonnegative solution $(u,c)$, we extend to the network some classical  estimates on the $L^p$ norms of the solutions of the Keller-Segel system on $\R^d$. By multiplying \eqref{KSunet} for $u^{p-1}$, $p>1$, and integrating over $\G$ we get
\[
\begin{split}
\f d{dt}\int_\G u^p dy&=\sum_{j=1}^m\kappa(e_j)\frac{d}{d t}\int_0^1\tilde u^p_j\,d\eta
=p\sum_{j=1}^m\kappa(e_j)\int_0^1\tilde u^{p-1}_j[\partial_{\eta\eta}\tilde u_j-\partial_\eta(\tilde u_j\partial_\eta\tilde c_j)]d\eta\\
&=p\sum_{i=1}^n\sum_{j\in E(v_i)}\kappa(e_j)u^{p-1}_j(t,v_i)[\f{\partial u_j}{\partial n}-u_j\f{\partial c_j}{\partial n}](t,v_i)\\
&\qquad-4\f{(p-1)}p\int_\G|\partial_yu^{p/2}|^2dy+2(p-1)\int_\G u^{p/2}\partial_yu^{p/2}\partial_yc\,dy\,.
\end{split}
\]
Using the continuity of $u$, the transmission conditions \eqref{transu} and \eqref{transc} and the H$\ddot{\text{o}}$lder inequality, we obtain
\beq\label{est:A}
\f d{dt}\int_\G u^p dy\le-2\f{(p-1)}p\int_\G|\partial_yu^{p/2}|^2dy+\f12\,p(p-1)\|\partial_yc(t)\|_{L^\infty(\G)}^2\int_\G u^p dy\,.
\eeq

Next, we shall estimate $\|\pd_y c(t)\|_{L^\infty(\G)}$ uniformly in time. By \eqref{elliptic}, we immediately get that $\tilde c\in C^2(0,1)$, for $j=1,\dots,m$. Integrating over $\G$ the equation on $c$ and observing that
\[
\int_\G \pd_{yy}c(t,y)\, dy =\sum_{j=1}^m\kappa(e_j)\int_0^1\pd_{\eta\eta}\tilde c_j(t,\eta)\,d\eta=\sum_{i=1}^n \sum\limits_{j\in  E(v_i)}\kappa(e_j) \f{\partial c_j}{\partial n}(t,v_i)=0\,,
\]
we get, by the  mass conservation property for $u$\,,
\beq\label{PSK3}
\alpha\int_\G c(t,y)dy=\int_\G u(t,y)dy=\int_\G u^0(y)dy=M\,.
\eeq
Integrating now the equation on $c_j$ over $[0,\eta]$, $\eta\in[0,1]$, we have for all $j=1,\dots, m$,
\[
\pd_\eta\tilde c_j(t,\eta)=\pd_\eta\tilde c_j(t,0) +\int_0^\eta [\alpha\,\tilde c_j(t,\sigma)-\tilde u_j(t,\sigma)]d\sigma\,,
\]
where $\pd_\eta\tilde c_j(t,0)$ is the internal derivative. Thus, since $u$ and $c$ are nonnegative, using \eqref{PSK3}, we obtain for $y=\Pi^\pm_j(\eta)$ and $j=1,\dots,m$,
\beq\label{PSK4}
\pd_y c_j(t,\Pi^\pm_j(0)) -\kappa_0^{-1}M\le\pd_y c_j(t,y)\le\pd_y c_j(t,\Pi^\pm_j(0))+\kappa_0^{-1}M\,.
\eeq
Hence, in order to bound $\|\pd_y c(t)\|_{L^\infty(\G)}$ it is sufficient to bound $\pd_y c(t)$ at the vertices $v_i$ of $\G$.

Integrating again the l.h.s. inequality in \eqref{PSK4} over $[0,y]$, $y\in e_j$, and using the positivity of $c$, we get
\[
y\,\pd_y c_j(t,\Pi^\pm_j(0)) -y\,\kappa_0^{-1}M\le c_j(t,y)\,,
\]
and after one more integration over $[0,1]$ we arrive at
\[
\f12\pd_y c_j(t,\Pi^\pm_j(0))-\f12\kappa_0^{-1}M\le\int_0^1\tilde c_j(t,\eta)\,d\eta\le(\alpha\,\kappa_0)^{-1}M\,,
\]
i.e.
\beq\label{upperderivatac}
\max_{i=1,\dots,n}\max_{j\in E(v_i)}\pd_y c_j(t,v_i)\le C_1(\alpha,\kappa_0)M\,.
\eeq
Plugging \eqref{upperderivatac} into the r.h.s. of \eqref{PSK4}, we arrive at the uniform in time upper bound for $\partial_yc(t)$
\beq\label{upperderivatac2}
\partial_yc(t)\le C_2(\alpha,\kappa_0)M\,.
\eeq

On the other hand, integrating twice as before the r.h.s. inequality in \eqref{PSK4}, we obtain
\beq\label{lowerderivatac}
-c_j(t,\Pi^\pm_j(0))-\f12\kappa_0^{-1}M\le \f12\pd_y c_j(t,\Pi^\pm_j(0))\,.
\eeq
Furthermore, thanks to \eqref{PSK3} and \eqref{upperderivatac2}, we observe that for all $j=1,\dots, m$ and $y=\Pi^\pm_j(\eta)\in\overline e_j$, it holds
\beq\label{Linftyboundc}
0\le c_j(t,y)=\int_0^1\tilde c_j(t,\sigma)\,d\sigma+\int_0^1\int_\sigma^\eta \pd_z\tilde c_j(t,z)dz\, d\sigma\le C_3(\alpha,\kappa_0)M\,.
\eeq
Plugging \eqref{Linftyboundc} into \eqref{lowerderivatac}, we arrive at a uniform in time lower bound for $\pd_y c_j(t,\Pi^\pm_j(0))$, $j=1,\dots,m$, giving a uniform in time lower bound for $\partial_yc(t)$ from the l.h.s. of \eqref{PSK4}. Resuming, we have obtained the existence of a constant $C(\alpha,\kappa_0)>0$ such that, in the maximal time interval of existence of the solution, it holds
\beq\label{stimaderivata}
\|c(t)\|_{W^{1,\infty}(\G)}\le C(\alpha,\kappa_0)M\,.
\eeq
To conclude, we replace \eqref{stimaderivata} into  \eqref{est:A} and we apply the iterative method introduced in \cite{AL} to get an uniform in time estimate of $\|u(t)\|_{L^\infty(\G)}$. The global existence of the nonnegative solution $(u,c)$ follows by the classical continuation in time argument.
\end{proof}
%%%%%%%%%%%%%%%%%%%%%%%%%%%%%%%%%%%%%%%%%%%%
\appendix
%%%%%%%%%%%%%%%%%%%%%%%%%%%%%%%%%%%%%%%%%%%%
\section{Appendix A: the heat kernel on networks}
\label{Appendix:A}
%%%%%%%%%%%%%%%%%%%%%%%%%%%%%%%%%%%%%%%%%%%%
This Appendix is devoted to the construction of the fundamental solution \eqref{def:H}-\eqref{def:L} of the heat equation \eqref{heat} on a weighted graph.

Let $\G$ be the considered weighted network. We recall that the edges $e_j$ are open. Let $x$ be a fixed point say of~$e_1$.  Let $y$ be  an arbitrary point of $\G\setminus V$ and $\pm a_j$, $j=1,\dots,m$, the two arcs containing~$y$. We consider the following function
\beq\label{H1}
H(t,x,y)=\delta_{1,j}\kappa^{-1}(e_1)G(t,d(x,y))+G(t,d(y,I(a_j))\star\phi_{a_j}(t,x)+G(t,d(y,I(-a_j))\star\phi_{-a_j}(t,x)\,,
\eeq
where $\phi_{a_j}$ and $\phi_{-a_j}$ are continuous functions with respect to $t\in[0,\infty)$, associated to the arc $\pm a_j$ respectively and to be determined for all $j=1,\dots,m$. Moreover, for functions $f$ and $g$ defined on $\R_+$, the $\star$ operator is defined as following
\[
(g\star f)(t)=\int_0^tg(s)f(t-s)ds\,,\qquad t\ge0\,.
\]

The following technical lemma will be useful in the sequel, (see \cite[Lemmas 1 and 2]{R}).
\begin{lemma}\label{teclem}
The following identities hold true
\begin{itemize}
\item[(i)] $\f\partial{\partial z}(G(t,z)\star f)_{|_{z=0^+}}=-\f12\,f(t)$ , if $f\in C^0([0,\infty))$ ;
\item[(ii)] $\f {z_1}t\,G(t,z_1)\star G(t,z_2)=G(t,z_1)$, for all $z_1\ne0$ and $z_2\in\R$ ;
\item[(iii)] $\f {z_1}t\,G(t,z_1)\star \f {z_2}t\,G(t,z_2)=\f {z_1+z_2}t\,G(t,z_1+z_2)$, for all $z_1,z_2\ne0$ .
\end{itemize}
\end{lemma}

It is easily seen that $H(\cdot,x,\cdot)$ is a solution of the heat equation on $(0,\infty)\times\G\setminus V$. Moreover, denoting $\theta_{\pm j}(t,x)$ the function $G(t,0)\star\phi_{\pm a_j}(t,x)$ and using Lemma \ref{teclem} (ii), it is easy to see that, for all $y\in\G\setminus V$, the function $H$ can be written also as
\beq\label{H2}
\begin{split}
H(t,x,y)=\delta_{1,j}\kappa^{-1}(e_1)G(t,d(x,y))&+t^{-1}d(y,I(a_j))\,G(t,d(y,I(a_j))\star\theta_{+j}(t,x)\\
&+t^{-1}d(y,I(-a_j))\,G(t,d(y,I(-a_j))\star\theta_{-j}(t,x)\,.
\end{split}
\eeq
Introducing the vector function $\Theta(t,x)=(\theta_{-1},\theta_{+1},\dots,\theta_{-m},\theta_{+m})^\top(t,x)$ and the $1\times2m$ matrix function $\Psi_j(t,y)=(\psi_{-1,j},\psi_{+1,j},\dots,\psi_{-m,j},\psi_{+m,j})(t,y)$ such that $\psi_{\pm l,j}=0$ if $l\ne j$ and
\[
\psi_{\pm j,j}(t,y)=t^{-1}d(y,I(\pm a_j))\,G(t,d(y,I(\pm a_j))\,,
\]
\eqref{H2} reads as
\beq\label{H3}
H(t,x,y)=\delta_{1,j}\kappa^{-1}(e_1)G(t,d(x,y))+\Psi_j(t,y)\star\Theta(t,x)\,.
\eeq

Next, for any fixed $x\in e_1$, we shall determine $\Theta(\cdot,x)$ in such a way that $H$ satisfies the following conditions on each ramifications nodes $v\in V$,
\begin{enumerate}
\item{\sl continuity condition} : there exists $b(t,x,v)\in\R$ such that $\displaystyle\lim_{y\in a_j\to v}H(t,x,y)=b(t,x,v)$, $j\in E(v)$;
\item{\sl transmission condition} : $\sum_{j\in E(v)}\kappa(e_j)\f\partial{\partial n}H(t,x,v)=0$.
\end{enumerate}
As a consequence, the resulting function $H$ shall satisfy all the properties $(i)$-$(vii)$ in Theorem \ref{th:R}.

Let $v\in V$ be an arbitrary fixed node. Because of the invariance of \eqref{H1} with respect to the parametrizations of $\pm a_j$, we can assume that $I(a_j)=v=T(-a_j)$, for all $j\in E(v)$, without loss of generality. Then, computing with respect to the parametrizaton of $a_j$ and passing to the limit $y\to v$ in \eqref{H1}, the continuity condition reads as
\beq\label{cont1}
H(t,x,I(a_j))=\delta_{1,j}\kappa^{-1}(e_1)G(t,d(x,I(a_j)))+G(t,0)\star\phi_{a_j}(t,x)+G(t,1)\star\phi_{-a_j}(t,x)=b(t,x,v)\,.
\eeq
By Lemma \ref{teclem} (ii), it holds
\[
G(t,1)\star\phi_{-a_j}(t,x)=\f1tG(t,1)\star\theta_{-j}(t,x)\,,
\]
and \eqref{cont1} can be written in term of the thetas functions as
\beq\label{cont2}
\delta_{1,j}\kappa^{-1}(e_1)G(t,d(x,I(a_j)))+\theta_{+j}(t,x)+\f1tG(t,1)\star\theta_{-j}(t,x)=b(t,x,v)\,,\qquad j\in E(v)\,.
\eeq
Multiplying \eqref{cont2} by $\kappa(e_j)$ and summing over all $j\in E(v)$, we get
\beq\label{eq:theta1}
\sum_{j\in E(v)}\delta_{1,j}G(t,d(x,I(a_j)))+\sum_{j\in E(v)}\kappa(e_j)\theta_{+j}(t,x)+\sum_{j\in E(v)}\kappa(e_j)\f1tG(t,1)\star\theta_{-j}(t,x)=b(t,x,v)\sum_{j\in E(v)}\kappa(e_j).
\eeq

On the other hand, according to the definition of the exterior normal derivatives and to Lemma~\ref{teclem}~(i), it holds true that
\[
\begin{split}
\f\partial{\partial n}H(t,x,I(a_j))=&-\delta_{1,j}\kappa^{-1}(e_1)\f1{2t}\,d(x,I(a_j))\,G(t,d(x,I(a_j)))\\
&\quad+\f12\phi_{a_j}(t,x)-\f1{2t}\,G(t,1)\star\phi_{-a_j}(t,x)\,,\qquad\qquad j\in E(v)\,.
\end{split}
\]
Again, by Lemma \ref{teclem} (ii), we easily obtain the identity
\beq\label{transm1}
\f\partial{\partial n}H(t,x,I(a_j))\star\f1{\sqrt{\pi t}}=-\delta_{1,j}\kappa^{-1}(e_1)G(t,d(x,I(a_j)))+\theta_{+j}(t,x)-\f1tG(t,1)\star\theta_{-j}(t,x)\,,\quad j\in E(v)\,.
\eeq
Since the transmission condition at the node $v$ is equivalent to the condition
\[
\sum_{j\in E(v)}\kappa(e_j)\f\partial{\partial n}H(t,x,v)\star\f1{\sqrt{\pi t}}=0\,,
\]
multiplying \eqref{transm1} by $\kappa(e_j)$ and summing over all $j\in E(v)$, we get
\beq\label{eq:theta2}
\sum_{j\in E(v)}\kappa(e_j)\theta_{+j}(t,x)=\sum_{j\in E(v)}\delta_{1,j}G(t,d(x,I(a_j)))+\sum_{j\in E(v)}\kappa(e_j)\f1tG(t,1)\star\theta_{-j}(t,x)\,.
\eeq
Plugging \eqref{eq:theta2} into \eqref{eq:theta1}, we have
\[
2\sum_{j\in E(v)}\delta_{1,j}G(t,d(x,I(a_j)))+2\sum_{j\in E(v)}\kappa(e_j)\f1tG(t,1)\star\theta_{-j}(t,x)=b(t,x,v)\sum_{j\in E(v)}\kappa(e_j)\,.
\]
The latter identity gives us the value $b(t,x,v)$ in term of the thetas functions, for any node $v\in V$. Defining $\kappa(v):=\sum_{j\in E(v)}\kappa(e_j)$ and plugging the obtained expression for $b(t,x,v)$ into \eqref{cont2}, we arrive at the identity
\[
\begin{split}
\delta_{1,j}\kappa^{-1}(e_1)G(t,d(x,I(a_j)))&+\theta_{+j}(t,x)+\f1tG(t,1)\star\theta_{-j}(t,x)\\
&=\f2{\kappa(v)}\sum_{i\in E(v)}\delta_{1,i}G(t,d(x,I(a_i)))+\f2{\kappa(v)}\sum_{i\in E(v)}\kappa(e_i)\f1tG(t,1)\star\theta_{-i}(t,x)\,,
\end{split}
\]
i.e. for any $j\in E(v)$ it holds
\beq\label{eq:theta3}
\begin{split}
\theta_{+j}(t,x)=&(\f{2\kappa(e_1)}{\kappa(v)}-1)\delta_{1,j}\,\kappa^{-1}(e_1)G(t,d(x,I(a_j)))+\sum_{i\in E(v),i\ne j}\f{2\kappa(e_i)}{\kappa(v)}\delta_{1,i}\,\kappa^{-1}(e_i)G(t,d(x,I(a_i)))\\
&+(\f{2\kappa(e_j)}{\kappa(v)}-1)\f1tG(t,1)\star\theta_{-j}(t,x)+\sum_{i\in E(v),i\ne j}\f{2\kappa(e_i)}{\kappa(v)}\f1tG(t,1)\star\theta_{-i}(t,x)\,.
\end{split}
\eeq
Taking into account that the sum of the first two terms in \eqref{eq:theta3} is not identically zero iff $1\in E(v)$ and using the transfert/reflection coefficients \eqref{def:trancoef}, \eqref{eq:theta3} can be written as
\beq\label{eq:theta4}
\theta_{+j}(t,x)=\epsilon_{(-a_1\to a_j)}\kappa^{-1}(e_1)G(t,d(x,I(a_1)))+\sum_{i\in E(v)}\epsilon_{(-a_i\to a_j)}\f1tG(t,1)\star\theta_{-i}(t,x)\,,\quad j\in E(v)\,.
\eeq

Now it is worth noticing that changing the roles between $a_j$ and $-a_j$ and making all the above computations with respect to the parametrization of $-a_j$, gives us that formula \eqref{eq:theta4} also holds true for $\theta_{-j}(t,x)$, i.e.
\beq\label{eq:theta5}
\theta_{-j}(t,x)=\epsilon_{(a_1\to-a_j)}\kappa^{-1}(e_1)G(t,d(x,I(-a_1)))+\sum_{i\in E(v)}\epsilon_{(a_i\to -a_j)}\f1tG(t,1)\star\theta_{+i}(t,x)\,,\quad j\in E(v)\,.
\eeq

Let denote the two arcs $\pm a_i$, supported by $e_i$ as $\pm i$ and the transfer/reflection coefficients $\epsilon_{(\pm a_i\to\pm a_j)}$ as $\epsilon_{\pm i,\pm  j}$. We recall that $\epsilon_{\pm i,\pm  j}\ne0$ iff $T(\pm i)=I(\pm j)$. Let denote ${\cal I}$ the ordered set $\{-1,1,\dots,-m,m\}$. Then, \eqref{eq:theta4}-\eqref{eq:theta5} reads both as
\[
\theta_l(t,x)=\kappa^{-1}(e_1)[\epsilon_{-1,l}G(t,d(x,I(1)))+\epsilon_{1,l}G(t,d(x,I(-1)))]+\sum_{k\in{\cal I}}\epsilon_{k,l}\f1tG(t,1)\star\theta_{k}(t,x)\,,\quad l\in{\cal I}\,,
\]
i.e. the function $\Theta$ satisfies the $2m$ system
\beq\label{2msystem}
\Theta(t,x)=\Lambda_1(t,x)+{\cal T}(t)\star \Theta(t,x)\,,
\eeq
where $\Lambda_1(t,x)$ is the $2m$ vector of components
\[
\lambda_l(t,x)=\kappa^{-1}(e_1)[\epsilon_{-1,l}G(t,d(x,I(1)))+\epsilon_{1,l}G(t,d(x,I(-1)))]\,,\qquad l\in{\cal I}\,,
\]
and ${\cal T}(t)$ is the $2m$ squared matrix ${\cal T}(t)=\f1tG(t,1)((\epsilon_{k,l})_{k,l\in{\cal I}})^\top$.

The solution of system \eqref{2msystem} is given by
\beq\label{Theta}
\Theta(t,x)=\Lambda_1(t,x)+\sum_{n=1}^\infty{\cal T}(t)^{\star\,n}\star\Lambda_1(t,x)\,,
\eeq
where ${\cal T}(t)^{\star\,n}$ is the iterated convolution $\underbrace{{\cal T}(t)\star\cdots\star{\cal T}(t)}_{n-times}$. Plugging \eqref{Theta} into \eqref{H3} we obtain for $x\in e_1$ and $y\in e_j$
\beq\label{H4}
H(t,x,y)=\delta_{1,j}\kappa^{-1}(e_1)G(t,d(x,y))+ \Psi_j(t,y)\star\Lambda_1(t,x)+\sum_{n=1}^\infty \Psi_j(t,y)\star{\cal T}(t)^{\star\,n}\star\Lambda_1(t,x)\,.
\eeq
Let observe that the term $\Psi_j(t,y)\star\Lambda_1(t,x)$ in \eqref{H4} is not equal to 0 iff $(\pm 1,\pm j)$ is a path. In the same way, the first term of the series in \eqref{H4} is the sum of $\psi_{\pm j,j}(t,y)\star{\cal T}(t)_{\pm j,l}\star\lambda_l(t,x)$ over all $l\in{\cal I}$ and that each of these terms is not equal to 0 iff $(\pm1,l,\pm j)$ forms a path. Reasoning iteratively and using Lemma \ref{teclem} (ii)-(iii) to compute the nonzero terms in the right hand side of  in \eqref{H4}, we arrive finally to the expression
\[
H(t,x,y)=\delta_{1,j}\,\kappa^{-1}(e_1)G(t,d(x,y))+\sum_{k\ge \rho(x,y)}\sum_{C\in C_{k+2}(x,y)}\kappa^{-1}(e_1)\,\epsilon(C)\,G(t,d_C(x,y))\,.
\]
Since initially $x$ is an arbitrarily fixed point of $e_1$, the previous formula holds true also for any $x\in e_i$, $i=1,\dots,m$, and by continuity for any $x$ vertex of $\G$, giving \eqref{def:H}-\eqref{def:L}.
%%%%%%%%%%%%%%%%
\section{Appendix B: Proof of Propositions \ref{prop:proprietaH} and \ref{prop:elliptic}}
\label{Appendix:B}
%%%%%%%%%%%%%%
\begin{proof}[Proof of Proposition \ref{prop:proprietaH}]
We shall prove \eqref{int H} first for any fixed $x\in\G\setminus V$. By the continuity of $H$ on $\G$, \eqref{int H} holds true for all $x\in\G$. Moreover, since $\int_\G H(t,x,y)dy$ is continuous with respect to $t$ and $\partial_t\int_\G H(t,x,y)dy=\int_\G\partial_{yy}H(t,x,y)dy=0$, it is enough to prove that $\lim_{t\to0^+}\int_\G H(t,x,y)dy=1$.

To begin, we have
\beq\label{intH}
\int_\G H(t,x,y)dy=\int_0^1G(t,|\xi-\eta|)d\eta+I(t,x)=1-\int_{\R\setminus(0,1)}G(t,|\xi-\eta|)d\eta+I(t,x)\,,
\eeq
where $\xi=(\Pi^\pm_i)^{-1}(x)$ and $\eta=(\Pi^\pm_j)^{-1}(y)$. The limit as $t\to0^+$ of the second term in the r.h.s. of \eqref{intH} is obviously 0. Concerning the remainder $I$, for any $x\in e_i$ fixed, there exists $\delta\in(0,1)$ such that $\xi\in(\delta,1-\delta)$. Then, we have
\[
|I|=|\sum_{j=1}^m\f{\kappa(e_j)}{\kappa(e_i)}\int_0^1\sum_{k\ge \rho(x,y)}\sum_{C\in C_{k+2}(x,y)}\epsilon(C)\,G(t,d_C(x,y))d\eta|
\le m\,\kappa_1\kappa_0^{-1}\sum_{k=0}^{+\infty}\f{e^{-(\delta^2+k^2)/4t}}{\sqrt{\pi\,t}}(\overline\epsilon\,m)^{k+1}\,,
\]
so that $\lim_{t\to0^+} I(t,x)=0$\,.

Next, let us notice that the function $L(t,x,y)$  in \eqref{def:L} is not a priori positive (because of $\epsilon(C)$, see \eqref{def:trancoef}-\eqref{def:epsC}), and so $H(t,x,y)$ too. Therefore,  \eqref{int H} does not give \eqref{est:norme1H}. However, it follows by \eqref{def:H} and \eqref{est:L} with $t\le1$ that
\beq\label{est:H}
|H(t,x,y)|\le K\,t^{-\f12}[1+\sum_{k=0}^{+\infty}(\overline\epsilon\,m)^k\,e^{-k^2/4}]=K\,t^{-\f12}\,,\qquad\forall (x,y)\in\G\times\G\,.
\eeq
In particular $|H(1,x,y)|\le K$. Applying the maximum principle for the heat equation on networks \cite{JvB2}, it holds
\beq\label{est:normeinftyH2}
|H(t,x,y)|\le K\,,\qquad\qquad\forall\ (t,x,y)\in[1,\infty)\times\G\times\G\,,
\eeq
and \eqref{est:norme1H} follows for $t\ge1$. On the other hand, for $t\in(0,1)$ and $x=\Pi^\pm_i(\xi)\in\overline e_i$ arbitrarily fixed, we have
\[
\begin{split}
\int_\G|H(t,x,y)|dy&\le\int_0^1G(t,|\xi-\eta|)d\eta+\int_\G|L(t,x,y)|dy\\
&\le 1+\sum_{j=1}^m\f{\kappa(e_j)}{\kappa(e_i)}\sum_{k=0}^{+\infty}(\overline\epsilon\,m)^{k+1}\,e^{-k^2/4}\int_0^1\f1{\sqrt{\pi\,t}}e^{-\eta^2/4t}d\eta\le 1+K\,.
\end{split}
\]
The proof of \eqref{est:norme1H} is now complete. From \eqref{est:H} and \eqref{est:normeinftyH2}, estimate \eqref{est:normeinftyH} follows too.

For the derivative of $H$ we have
\[
\partial_yH(t,x,y)=\delta_{i,j}\,\kappa^{-1}(e_i)\partial_yG(t,d(x,y))+\partial_yL(t,x,y)\,,\qquad\forall\ (t,x,y)\in(0,\infty)\times\G\times\G\,,
\]
where
\[
\partial_yL(t,x,y)=\sum_{k\ge \rho(x,y)}\sum_{C\in C_{k+2}(x,y)}\kappa^{-1}(e_i)\,\epsilon(C)\partial_yG(t,d_C(x,y))\,.
\]
%Let $x\in\G$ be fixed, $y=\Pi^\pm_j(\eta)\in \overline e_j$, $j=1,\dots,m$, and define $f(z)=ze^{-z^2}$.
Observing that for any $C\in C_{k+2}(x,y)$ we have $d_C(x,y)=(1-\xi)+k+\eta$, we get the estimate
\beq\label{est:normeinftydevH1}
|\partial_yH(t,x,y)|\le K\,t^{-1}[1+\sum_{k=0}^{+\infty}(\overline\epsilon\,m)^k\,e^{-k^2/4t}+\sum_{k=0}^{+\infty}(\overline\epsilon\,m)^k\f k{2\sqrt t}\,e^{-k^2/4t}]\,,\qquad\forall\ (t,x,y)\in(0,\infty)\times\G\times\G\,.
\eeq
Moreover,
\beq\label{est:normeinftydevH2}
\sum_{k=0}^{+\infty}(\overline\epsilon\,m)^k\f k{2\sqrt t}\,e^{-k^2/4t}\le\f 1{2\sqrt t}\,e^{-1/4t}\sum_{k=1}^{+\infty}(\overline\epsilon\,m)^kk\,e^{-(k-1)^2/4t}\,,
%\le C\sum_{k=1}^{+\infty}(\overline\epsilon\,m)^kk\,e^{-(k-1)^2/4t}\,,
\eeq
and \eqref{est:normeinftydevH} follows for $t\le1$, from \eqref{est:normeinftydevH1} and \eqref{est:normeinftydevH2}. For $t>1$, we obtain \eqref{est:normeinftydevH} as before, observing that $|\partial_yH(1,x,y)|\le K$ and applying again the maximum principle.

It remains to prove \eqref{est:norme1derH} for $t\in(0,1)$, since for $t\ge1$ \eqref{est:norme1derH} follows by the obtained estimate $\|\partial_yH(t)\|_{L^\infty(\G\times\G)}\le K$. Let $x=\Pi^\pm_i(\xi)\in\overline e_i$ be arbitrarily fixed. We have
\[
\begin{split}
\int_\G|\partial_yH(t,x,y)|dy&\le\int_0^1G(t,|\xi-\eta|)\f{|\xi-\eta|}{2t}d\eta+\int_\G|\partial_yL(t,x,y)|dy\\
&\le K\,t^{-1/2}+C\sum_{j=1}^m\f{\kappa(e_j)}{\kappa(e_i)}\sum_{k=0}^{+\infty}(\overline\epsilon\,m)^{k+1}\f{(1-\xi)}te^{-(1-\xi)^2/4t}\,e^{-k^2/4}\int_0^1e^{-\eta^2/4t}\f{d\eta}{2\sqrt t}\\
&\quad+C\sum_{j=1}^m\f{\kappa(e_j)}{\kappa(e_i)}\sum_{k=0}^{+\infty}(\overline\epsilon\,m)^{k+1}\,e^{-k^2/4}
\f1{\sqrt{t}}\int_0^1e^{-\eta^2/4t}\f\eta{4t}d\eta\\
&\quad+C\sum_{j=1}^m\f{\kappa(e_j)}{\kappa(e_i)}\sum_{k=0}^{+\infty}(\overline\epsilon\,m)^{k+1}\,\f kt\,e^{-k^2/4t}\int_0^1e^{-\eta^2/4t}\f{d\eta}{2\sqrt t}\\
&\le K\,t^{-1/2}+\f C{\sqrt t}\sum_{k=0}^{+\infty}(\overline\epsilon\,m)^{k}\,e^{-k^2/4}+\f C{\sqrt t}\sum_{k=0}^{+\infty}(\overline\epsilon\,m)^k\f k{2\sqrt t}\,e^{-k^2/4t}\,.
\end{split}
\]
Arguing as in \eqref{est:normeinftydevH2}, we arrive at
\[
\sup_{x\in\G}\|\partial_yH(t,x,\cdot)\|_{L^1(\G)}\le K\,t^{-1/2}\,,\qquad\forall\ t\in(0,1)\,.
\]

Finally, the estimate $\sup_{y\in\G}\|\partial_yH(t,\cdot,y)\|_{L^1(\G)}\le K\,t^{-1/2}$, for $t\in(0,1)$, can be obtained as above, changing the role between $\xi$ and $\eta$, and the proof is complete. \\
\end{proof}
\begin{proof}[Proof of Proposition \ref{prop:elliptic}]
We say that  a function $w\in H^1(\G)$ is a weak solution of \eqref{elliptic} if it a solution of the variational formulation of \eqref{elliptic}
\beq\label{elliptic1}
a(w,\phi)=(z,\phi)_{L^2(\G)}\,, \qquad\forall\ \phi\in H^1(\G)\,,
\eeq
where  $a(w,\phi):=\int_\G(w'(x)\phi'(x)+\alpha\,w(x)\phi(x))\,dx$ is a continuous,  coercive bilinear form  on $H^1(\G)$. A  standard application of the Lax-Milgram's Theorem gives the existence and uniqueness of the solution of \eqref{elliptic1}. The inequality  \eqref{maxprinc} is consequence of the maximum principle  (see \cite{N3}).  To prove \eqref{estderiv}, observe that   the  equation and \eqref{maxprinc}   imply that
$w''_j$ is bounded for all $j=1,\dots,m$ and
\[
\|w''\|_{L^\infty(\G)}\le \|z\|_{L^\infty(\G)}\,.
\]
By the previous estimate and the Morrey's inequality we get the estimate \eqref{estderiv}.
\end{proof}
%
%Let $G$ be the head kernel in dimension $d\ge1$, i.e.
%\[
%G(x,t)=\f1{(4\pi\,t)^{\f d2}}e^{-\f{|x|^2}{4t}}\,,\qquad x\in\R^d\,,\ t>0\,.
%\]
%Then, for any $p\in[1,\infty]$ it holds
%\[
%\|G(t)\|_{L^p}=\f{C(p,d)}{t^{\f d2(1-\f1p)}}\,,\qquad t>0\,.
%\]
%Indeed, for $p=+\infty$ it is obvious, and for $p\in[1,\infty)$ we have
%\[
%\int_{\R^d}G^p(x,t)\,dx=\f1{(4\pi\,t)^{\f d2p}}\left(\f{2\sqrt t}{\sqrt p}\right)^d\int_{\R^d}e^{-|y|^2}dy\,.
%\]
%Moreover,
%\[
%\nabla G(x,t)=\f1{(4\pi\,t)^{\f d2}}e^{-\f{|x|^2}{4t}}\left(-\f x{2t}\right)=\f1{(4\pi\,t)^{\f d2}}\f1{\sqrt t}f\left(\f x{2\sqrt t}\right)\,,\qquad x\in\R^d\,,\ t>0\,.
%\]
%with
%\[
%f(y)=-y\,e^{-|y|^2}\,,\qquad y\in\R^d\,.
%\]
%Then, in the same way, we have for all $p\in[1,\infty]$
%\[
%\|\nabla G(t)\|_{L^p}=\f{C(p,d)}{t^{\f d2(1-\f1p)+\f12}}\,,\qquad t>0\,.
%\]
%%%%%%%%%%%%%%%%%%%%%%%%%%%%%%%%%%%%%
\begin{acknowledgment} The authors would like to thanks Jean-Pierre Roth for helpful discussions.
 This research was initiated when the first author was visiting the Universit\'e d'Evry Val d'Essonne; the financial support and kind hospitality are gratefully acknowledged. The second author acknowledges the support of the french ``ANR blanche'' project Kibord : ANR-13-BS01-0004  and of ``Progetto
 Gnampa 2015: Network e controllabilit\`a".
\end{acknowledgment}
%%%%%%%%%%%%%%%%%%%%%%%%%%%%%%%%%%%%
%%%%%% BIBLIO %%%%%%%%%%%%%%%%%%%%%%
%%%%%%%%%%%%%%%%%%%%%%%%%%%%%%%%%%%%
\bibliographystyle{amsplain}

%%%%%%%%%%%%%%%%%%%%%%%%%%%%%%%%%%%%
\noindent ${^a}$ Dipartimento di Scienze di Base e Applicate per l'Ingegneria, \\
``Sapienza'' Universit{\`a}  di Roma, \\
via Scarpa 16, 0161 Roma (Italy) \\
e-mail:camilli@dmmm.uniroma1.it

%%%%%%%%%%%%%%%%%%%%%%%%%%%%%%%%%%%%
\noindent ${^b}$ LaMME - UMR 8071 \\
Universit\'e d'Evry Val d'Essonne, \\
23 Bd. de France, F--91037 Evry Cedex, France\\
e-mail:lucilla.corrias@univ-evry.fr
%%%%%%%%%%%%%%%%%%%%%%%%%%%%%%%%%%%%%%
\end{document}